\documentclass[11pt,leqno]{article}
\usepackage{amsmath,amssymb,amsthm,latexsym}
\usepackage{indentfirst}
\usepackage{amsmath}
\newtheorem{Definition}{\bf \large Definition}[section]
\newtheorem{Theorem}{\bf \large Theorem}[section]
\newtheorem{PROPOSITION}{\bf \large Proposition}[section]

\newtheorem{Remark}{\bf \large Remark}[section]

\title{\textbf{Wintgen ideal submanifolds with a low-dimensional integrable distribution (I)}}
\author{Tongzhu Li, Xiang Ma, Changping Wang}
\date{}

\begin{document}
\maketitle

\footnotetext{T. Li, X. Ma, C.P. Wang are partially supported
by the grant No. 11171004 of NSFC; X. Ma is partially supported by the grant No. 10901006 of NSFC.}

\begin{abstract}
A submanifold in space forms satisfies the well-known DDVV inequality due to De Smet, Dillen, Verstraelen and Vrancken. The submanifold attaining equality in the DDVV inequality at every point is called Wintgen ideal submanifold.
As conformal invariant objects, Wintgen ideal submanifolds are studied in this paper using the framework of M\"{o}bius geometry. We classify Wintgen ideal submanfiolds of dimension $m>2$ and arbitrary codimension
when a canonically defined $2$-dimensional distribution $\mathbb{D}$ is integrable. Such examples come from cones, cylinders, or rotational submanifolds over super-minimal surfaces in spheres, Euclidean spaces,
or hyperbolic spaces, respectively.
\end{abstract}

\vskip 1 cm
\section{Introduction}
A basic idea in submanifold theory is to find certain universal
inequalities (pointwise or global ones) between various invariants (intrinsic and extrinsic), then characterize and
classify the optimal submanifolds attaining equality in such inequalities.

The simplest example of such \emph{pointwise} inequalities is
that for mean curvature $H$ and Gaussian curvature $K$ of a
surface in $\mathbb{R}^3$, there is always $H^2\ge K$, and the
equality holds true exactly at those umbilic points. This was generalized by Chen \cite{ch96} to other space forms and to arbitrary codimensional case.
(Another universal inequality posed by Chen \cite{ch93} motivated a series of investigation on submanifolds attaining the equality in Chen's inequality.)

De Smet, Dillen, Verstraelen and Vrancken
proposed in 1999 a strengthened inequality \cite{DDVV1}
involving the scalar, mean curvature and the norm of th normal curvature tensor as below.
Let $f:M^m\longrightarrow \mathbb{Q}^{m+p}_c$ be an isometric immersion of an $m-$dimensional Riemannian manifold into a space form of dimension $m+p$
and constant sectional curvature $c$. Let $R$ (resp. $R^{\bot}$) be the Riemannian curvature tensor (resp. the normal curvature tensor) of $f$. Their conjectured that at any point,
\begin{equation}\label{ineq}
DDVV ~~inequality:~~~~~~~~~~~s\leq c+||H||^2-s^{\bot}.
\end{equation}
Here $H$ denotes the mean curvature of $f$, and
\[
s=\frac{2}{m(m-1)}\sum_{1\leq i<j\leq n}\langle R(e_i,e_j)e_j,e_i\rangle ,~~~
s^{\bot}=\frac{2}{m(m-1)}||R^{\perp}||.
\]
This so-called DDVV conjecture was proved in 2008 by J. Ge, Z. Tang \cite{Ge} and Z. Lu \cite{Lu1} independently.

Moreover, in \cite{Ge} the pointwise structure of the second fundamental form of $f$ that attains equality was determined. It was shown that equality holds at
$x\in M^m$ if, and only if, there exists an orthonormal basis $\{e_1,\cdots,e_m\}$ in the tangent space $T_xM^m$ and an orthonormal basis $\{n_1,\cdots,n_p\}$ in the normal space $T_x^{\bot}M^m$,
such that the shape operators $\{A_{n_i},i=1,\cdots,m\}$ have the form
\begin{equation}\label{form1}
A_{n_1}=
\begin{pmatrix}
\lambda_1 & \mu_0 & 0 & \cdots & 0\\
\mu_0 & \lambda_1 & 0 & \cdots & 0\\
0  & 0 & \lambda_1 & \cdots & 0\\
\vdots & \vdots & \vdots & \ddots & \vdots\\
0  & 0 & 0 & \cdots & \lambda_1
\end{pmatrix},
A_{n_2}=
\begin{pmatrix}
\lambda_2+\mu_0 & 0 & 0 & \cdots & 0\\
0 & \lambda_2-\mu_0 & 0 & \cdots & 0\\
0  & 0 & \lambda_2 & \cdots & 0\\
\vdots & \vdots & \vdots & \ddots & \vdots\\
0  & 0 & 0 & \cdots & \lambda_2
\end{pmatrix},
\end{equation}
and
$$A_{n_3}=\lambda_3I_p,~~~~ A_{n_r}=0, r\ge 4.$$

Wintgen first proved the inequality \eqref{ineq} for surfaces $f:M^2\to S^4$. The equality holds at $x\in M^2$ if and only if the curvature ellipse
of $M^2$ in $S^4$ at $x$ is a circle \cite{wint, gr}.
$f:M^2\to S^4$ is called a \emph{super-conformal} surface if this holds true at every point. It is well-known that such surfaces
correspond to images of complex curves in $\mathbb{CP}^3$ via the Penrose twistor projection $\pi:\mathbb{ CP}^3\longrightarrow \mathbb{S}^4$. According to the suggestion of \cite{chen10,ml}, we make the following definition.

\begin{Definition}
A submanifold $f:M^m\longrightarrow \mathbb{Q}^{m+p}_c$ ia called a \emph{Wintgen ideal submanifold} if the equality is satisfied in \eqref{ineq} at every point of $M^m$.
\end{Definition}

We briefly review known results on the classification of Wintgen ideal submanifolds. It was shown in \cite{gr} that $f:M^2\longrightarrow \mathbb{Q}^{2+p}_c$ is a Wintgen ideal surface
if and only if the ellipse of curvature of $f$ at $x$ is a circle. When being also minimal surfaces in the specific space form, they were already known as \emph{super-minimal surfaces}. Such examples are abundant.

For three dimensional Wintgen ideal subanifolds, when they are minimal, they belong to the class of (three dimensional) austere submanifolds. The later was classified locally by Bryant for Euclidean submanifolds \cite{br}, by Dajczer and Florit in the unit sphere \cite{Dajczer}, by Choi and Lu \cite{Lu,Lu2} in hyperbolic space.

Finally, in \cite{Dajczer1}, Dajczer and
Tojeiro provided a parametric construction of Wintgen ideal submanifolds of codimension two and arbitrary dimension in terms of minimal surfaces in the Euclidean space.

An important observation of Dajczer and Tojeiro in \cite{Dajczer1} is that the inequality as well as the equality case (and the class of Wintgen ideal submanifolds) are conformally invariant property. This follows from the observation
in \cite{DDVV} that inequality \eqref{ineq} holds at a point $x\in M^m$ if and only if
\[
\sum_{\alpha,\beta=1}^p||[\bar{A}_{\alpha},\bar{A}_{\beta}]||^2
\leq\left(\sum_{\alpha=1}^p||\bar{A}_{\alpha}||\right)^2
\]
is satisfied for the traceless shape operators $\bar{A}_1,\cdots,\bar{A}_p$ at $x\in M^m$, whereas $\{\bar{A}_i\}$ are conformal invariant objects (up to a scalar factor). It follows that Wintgen ideal submanifolds in the sphere $\mathbb{S}^{m+p}$ or hyperbolic space $\mathbb{H}^{m+p}$ are the pre-image of a stereographic projection of Wintgen ideal submanifolds in $\mathbb{R}^{m+p}$.

Thus it is appropriate to put the study of Wintgen ideal submanifolds in
the framework of M\"{o}bius geometry. For the same reason it is no restriction when we describe them in the Euclidean space.
This is exactly our main goal in this paper.

As the main result, we give a classification of Wintgen ideal submanifolds with integrable canonical distribution $\mathbb{D}=Span\{e_1,e_2\}$.
It is clear from \eqref{form1} that $\mathbb{D}$ is well-defined when the submanifold is nowhere totally umbilic.
\begin{Theorem}\label{the1}
Let $f:M^m\longrightarrow \mathbb{R}^{m+p} (m\geq3)$ be a Wintgen ideal submanifold without umbilic points. If the canonical distribution $\mathbb{D}=span\{e_1,e_2\}$
 is integrable, then locally $f$ is M\"{o}bius equivalent to
\begin{description}
\item (i) a cone over a super-minimal surface in ${\mathbb S}^{2+p}$;
\item (ii) or a cylinder over a super-minimal surface in ${\mathbb R}^{2+p}$;
\item (iii) or a rotational submanifold over a super-minimal surface in ${\mathbb H}^{2+p}$.
\end{description}
\end{Theorem}

This paper is organized as follows. In section 2, we review the
elementary facts about M\"{o}bius geometry for submanifolds in
$\mathbb{R}^{m+p}$. In section 3, we describe the construction of Wintgen ideal submanifolds as cylinders, cones, or rotational submanifolds. In section 4, we give the proof of our main theorem.

 \vskip 1 cm
\section{Submanifold theory in M\"obius geometry}

In this section we briefly review the theory of submanifolds
in M\"obius geometry. For details we refer to \cite{CPWang} and \cite{liu}.

Let $\mathbb{R}^{m+p+2}_1$ be the Lorentz space with inner
product $\langle \cdot,\cdot\rangle $ defined by
\[
\langle Y,Z\rangle =-Y_0Z_0+Y_1Z_1+\cdots+Y_{m+p+1}Z_{m+p+1},
\]
where $Y=(Y_0,Y_1,\cdots,Y_{m+p+1}),Z=(Z_0,Z_1,\cdots,Z_{m+p+1})\in
\mathbb{R}^{m+p+2}$.

Let $f:M^m\rightarrow \mathbb{R}^{m+p}$ be a submanifold without umbilics and
assume that $\{e_i\}$ is an orthonormal basis with respect to the
induced metric $I=df\cdot df$ with $\{\theta_i\}$ the dual basis.
Let $\{n_{r}|1\le r\le p\}$ be a local orthonormal basis for the
normal bundle. As usual we denote the second fundamental form and
the mean curvature of $f$ as
\[
II=\sum_{ij,\gamma}h^{r}_{ij}\theta_i\otimes\theta_j n_{r},
~~H=\frac{1}{m}\sum_{j,r}h^{r}_{jj}n_{r}
=\sum_{r}H^{r}n_{r}.
\]
We define the M\"{o}bius position vector $Y:
M^m\rightarrow \mathbb{R}^{m+p+2}_1$ of $f$ by
\[
Y=\rho\left(\frac{1+|f|^2}{2},\frac{1-|f|^2}{2},f\right),~~
~~\rho^2=\frac{m}{m-1}\left|II-\frac{1}{m} tr(II)I\right|^2~.
\]
It is known that $Y$ is a well-defined canonical lift of $f$. Two submanifolds $f,\bar{f}: M^m\rightarrow \mathbb{R}^{m+p}$
are M\"{o}bius equivalent if there exists $T$ in the Lorentz group
$\mathbb{O}(m+p+1,1)$ in $\mathbb{R}^{m+p+2}_1$ such that $\bar{Y}=YT.$ It follows
immediately that
\[
g=\langle dY,dY\rangle =\rho^2 dx\cdot dx
\]
is a M\"{o}bius invariant, called the M\"{o}bius metric of $f$.

Let $\Delta$ be the Laplacian with respect to $g$. Define
\[
N=-\frac{1}{m}\Delta Y-\frac{1}{2m^2}\langle \Delta Y,\Delta Y\rangle Y,
\]
which satisfies
\[
\langle Y,Y\rangle =0=\langle N,N\rangle , ~~\langle N,Y\rangle =1~.
\]
Let $\{E_1,\cdots,E_m\}$ be a local orthonormal basis for $(M^m,g)$
with dual basis $\{\omega_1,\cdots,\omega_m\}$. Write
$Y_j=E_j(Y)$. Then we have
\[
\langle Y_j,Y\rangle =\langle Y_j,N\rangle =0, ~\langle Y_j,Y_k\rangle =\delta_{jk}, ~~1\leq j,k\leq m.
\]
We define
\[
\xi_r=H^r\left(\frac{1+|f|^2}{2},\frac{1-|f|^2}{2},f\right)
+\left(f\cdot n_r,-f\cdot n_r,n_r\right).
\]
Then $\{\xi_{1},\cdots,\xi_p\}$ be the orthonormal basis of the
orthogonal complement of $\mathrm{Span}\{Y,N,Y_j|1\le j\le m\}$.
And $\{Y,N,Y_j,\xi_{r}\}$ form a moving frame in $R^{m+p+2}_1$
along $M^m$.
\begin{Remark}
Geometrically, $\xi_r$ corresponds to the unique sphere tangent to $M_m$ at one point $x$ with normal vector $n_r$ and the same mean curvature $H^r(x)$. We call $\{\xi_r\}$ the mean curvature spheres of $M^m$.
\end{Remark}
We will use the following range of indices in this section: $1\leq
i,j,k\leq m; 1\leq r,s\leq p$. We can write the structure equations
as below:
\begin{eqnarray*}
&&dY=\sum_i \omega_i Y_i,\\
&&dN=\sum_{ij}A_{ij}\omega_i Y_j+\sum_{i,\gamma} C^{\gamma}_i\omega_i \xi_{r},\\
&&d Y_i=-\sum_j A_{ij}\omega_j Y-\omega_i N+\sum_j\omega_{ij}Y_j
+\sum_{j,\gamma} B^{\gamma}_{ij}\omega_j \xi_{r},\\
&&d \xi_{r}=-\sum_i C^{r}_i\omega_i Y-\sum_{ijr}\omega_i
B^{r}_{ij}Y_j +\sum_{s} \theta_{rs}\xi_{s},
\end{eqnarray*}
where $\omega_{ij}$ are the connection $1$-forms of the M\"{o}bius
metric $g$ and $\theta_{rs}$ the normal connection $1$-forms. The
tensors
\[
{\bf A}=\sum_{ij}A_{ij}\omega_i\otimes\omega_j,~~ {\bf
B}=\sum_{ijr}B^{r}_{ij}\omega_i\otimes\omega_j \xi_{r},~~
\Phi=\sum_{jr}C^{r}_j\omega_j \xi_{r}
\]
are called the Blaschke tensor, the M\"{o}bius second fundamental
form and the M\"{o}bius form of $x$, respectively. The covariant
derivatives of $C^{r}_i, A_{ij}, B^{r}_{ij}$ are defined by
\begin{eqnarray*}
&&\sum_j C^{r}_{i,j}\omega_j=d C^{r}_i+\sum_j C^{r}_j\omega_{ji}
+\sum_{s} C^{s}_j\theta_{sr},\\
&&\sum_k A_{ij,k}\omega_k=d A_{ij}+\sum_k A_{ik}\omega_{kj}+\sum_k A_{kj}\omega_{ki},\\
&&\sum_k B^{r}_{ij,k}\omega_k=d B^{r}_{ij}+\sum_k
B^{r}_{ik}\omega_{kj} +\sum_k B^{r}_{kj}\omega_{ki}+\sum_{s}
B^{s}_{ij}\theta_{sr}.
\end{eqnarray*}
The integrability conditions for the structure equations are given by
\begin{eqnarray}
&&A_{ij,k}-A_{ik,j}=\sum_{r}B^{r}_{ik}C^{r}_j-B^{r}_{ij}C^{r}_k,\label{equa1}\\
&&C^{r}_{i,j}-C^{r}_{j,i}=\sum_k(B^{r}_{ik}A_{kj}-B^{r}_{jk}A_{ki}),\label{equa2}\\
&&B^{r}_{ij,k}-B^{r}_{ik,j}=\delta_{ij}C^{r}_k-\delta_{ik}C^{r}_j,\label{equa3}\\
&&R_{ijkl}=\sum_{r}B^{r}_{ik}B^{r}_{jl}-B^{r}_{il}B^{r}_{jk}
+\delta_{ik}A_{jl}+\delta_{jl}A_{ik}
-\delta_{il}A_{jk}-\delta_{jk}A_{il},\label{equa4}\\
&&R^{\perp}_{rs kl}=\sum_k
B^{r}_{ik}B^{s}_{kj}-B^{s}_{ik}B^{r}_{kj}. \label{equa5}
\end{eqnarray}
Here $R_{ijkl}$ denote the curvature tensor of $g$,
$\kappa=\frac{1}{n(n-1)}\sum_{ij}R_{ijij}$ is its normalized
M\"{o}bius scalar curvature. It follows from \eqref{equa4} that
the Ricci tensor of $g$ satisfies
\begin{equation}\label{equa6}
R_{ij}:=\sum_k R_{ikjk}=-\sum_{kr} B^{r}_{ik}B^{r}_{kj}+(tr{\bf
A})\delta_{ij}+(m-2)A_{ij}.
\end{equation}
Other restrictions on tensors $\bf A, B$ are
\begin{eqnarray}
&&\sum_j B^{r}_{jj}=0, ~~~\sum_{ijr}(B^{r}_{ij})^2=\frac{m-1}{m}, \label{equa7}\\
&& tr{\bf A}=\sum_j A_{jj}=\frac{1}{2m}(1+m^2\kappa).\label{equa8}
\end{eqnarray}
 We know that all coefficients in the
structure equations are determined by $\{g, {\bf B}\}$ and the
normal connection $\{\theta_{\alpha\beta}\}$.

\section{Examples of Wintgen ideal submanifolds}
In this section we will use minimal Wintgen ideal submanifolds
in space forms $\mathbb{R}^n, \mathbb{S}^n$ or $\mathbb{H}^n$ to construct other general examples. Note that being a minimal submanifold is not a conformal invariant property.

We remark that a minimal Wintgen ideal submanifold in any space form is an austere submanifold of rank two, and the converse is also true. Super-minimal surfaces in space forms are special examples and there are plenty of them, including all minimal 2-spheres in $\mathbb{S}^n$ and all complex curves in $\mathbb{C}^n=\mathbb{R}^{2n}$.

\begin{Definition}
Let $u: M^r\longrightarrow \mathbb{R}^{r+p}$ be an immersed  submanifold. We define
\emph{the cylinder over $u$} in $\mathbb{R}^{m+p}$ as
\[
f=(u,id):M^r\times \mathbb{R}^{m-r}\longrightarrow
\mathbb{R}^{r+p} \times \mathbb{R}^{m-r}=\mathbb{R}^{m+p},
\]
where $id:\mathbb{R}^{m-r}\longrightarrow \mathbb{R}^{m-r}$ is the identity map.
\end{Definition}

\begin{PROPOSITION}\label{31}
Let $u: M^r\longrightarrow \mathbb{R}^{r+p}$ be an immersed submanifold. Then the cylinder
$f=(u,id):M^r\times \mathbb{R}^{m-r}\longrightarrow \mathbb{R}^{m+p}$ is a Wintgen ideal submanifold if and only if $u: M^r\longrightarrow \mathbb{R}^{r+p}$ is a
minimal Wintgen ideal submanifold.
\end{PROPOSITION}

\begin{proof} Let $\eta_1,\cdots,\eta_p$ be an orthonormal frame in the normal bundle of $u$ in $\mathbb{R}^{r+p}$. Then $e_r=(\eta_r,\vec{0})\in
\mathbb{R}^{m+p}$ is such a frame of $f$. The first and second
fundamental forms $I,II$ of $f$ are related with corresponding forms $I_u,II_u$ of $u$ by
\begin{equation}\label{II31}
I=I_u+I_{\mathbb{R}^{m-r}}, \;\; II=II_u,
\end{equation}
where $I_{\mathbb{R}^{m-r}}$ denotes the standard metric of $\mathbb{R}^{m-r}$.
Clearly $f$ is a Wintgen ideal submanifold if and only if $u$ is a minimal Wintgen ideal submanifold. This completes the proof to Proposition \ref{31}.
\end{proof}

\begin{Remark}
The M\"{o}bius position vector $Y:M^r\times \mathbb{R}^{m-r}\longrightarrow \mathbb{R}^{m+p+2}_1$ of the cylinder $f$ is
\begin{equation}\label{mop1}
\begin{split}
&Y=\rho_0\left(\frac{1+|u|^2+|y|^2}{2},
\frac{1-|u|^2-|y|^2}{2},u,y\right),\\
&Y:M^r\times \mathbb{R}^{m-r}\longrightarrow \mathbb{R}^{r+p+2}_1\times \mathbb{R}^{m-r},
\end{split}
\end{equation}
where $\rho_0=\frac{m}{m-1}(|II_u|^2-mH_u^2): M^r\longrightarrow \mathbb{R}$, and $y:\mathbb{R}^{m-r}\longrightarrow \mathbb{R}^{m-r}$ is the identity map.
\end{Remark}

\begin{Definition}
Let $u:M^r\longrightarrow \mathbb{S}^{r+p}\subset \mathbb{R}^{r+p+1}$ be an immersed submanifold.
We define \emph{the cone over $u$} in $\mathbb{R}^{m+p}$ as
\begin{equation*}
\begin{split}
&f:R^+\times \mathbb{R}^{m-r-1}\times M^r\longrightarrow \mathbb{R}^{m+p},\\
&~~~~~~f(t,y,u)=(y,tu),
\end{split}
\end{equation*}
\end{Definition}

\begin{PROPOSITION}\label{32}
Let $u: M^r\longrightarrow \mathbb{S}^{r+p}$ be an immersed submanifold. Then the cone
$f=(y,tu):R^+\times \mathbb{R}^{m-r-1}\times M^r\longrightarrow \mathbb{R}^{m+p}$ is a Wintgen ideal submanifold if and only if $u$ is a minimal Wintgen ideal submanifold in $\mathbb{S}^{r+p}$.
\end{PROPOSITION}
\begin{proof}
The first and second
fundamental forms of $f$ are, respectively,
\[
I=t^2I_u+I_{\mathbb{R}^{m-r}}, \;\; II=t~II_u,
\]
where $I_u,II_u,I_{\mathbb{R}^{m-r}}$ are understood as before.
The conclusion follows easily.
\end{proof}

\begin{Remark}
The M\"{o}bius position vector $Y:R^+\times \mathbb{R}^{m-r-1}\times M^r\longrightarrow \mathbb{R}^{m+p+2}_1$ of the cone $f$ is
\begin{equation*}
Y=\rho_0\left(\frac{1+t^2+|y|^2}{2t},\frac{1-t^2-|y|^2}{2t},y,u\right),
\end{equation*}
where $\rho_0=\frac{m}{m-1}(|II_u|^2-mH_u^2): M^r\longrightarrow \mathbb{R}$, and $y:\mathbb{R}^{m-r-1}\longrightarrow \mathbb{R}^{m-r-1}$ is the identity map.
Let $$\mathbb{H}^{m-r}=\{(y_0,y)\in \mathbb{R}^{m-r+1}|-y_0^2+|y|^2=-1, y_0\geq 1\}\cong R^+\times \mathbb{R}^{m-r-1},$$
then $(\frac{1+t^2+|y|^2}{2t},\frac{1-t^2-|y|^2}{2t},y):R^+\times \mathbb{R}^{m-r-1}=\mathbb{H}^{m-r}\to \mathbb{H}^{m-r}$ is nothing else but the identity map. And the M\"{o}bius position vector of the cone $f$ is
\begin{equation}\label{mop2}
Y=\rho_0(id,u):\mathbb{H}^{m-r}\times M^r\to \mathbb{H}^{m-r}\times \mathbb{S}^{r+p}\subset \mathbb{R}^{m+p+2}_1,
\end{equation}
where $\rho_0\in C^{\infty}(M^r)$ and $id:\mathbb{H}^{m-r}\to \mathbb{H}^{m-r}$ is a identity map.
\end{Remark}

\begin{Definition}
Let $\mathbb{R}^{r+p}_+=\{(x_1,\cdots,x_{r+p})\in \mathbb{R}^{r+p}|x_{r+p}>0\}$
be the upper half-space endowed with the standard
hyperbolic metric
\[
ds^2=\frac{1}{x_{r+p}^2}\sum_{i=1}^{r+p} dx_i^2~.
\]
Let $u=(x_1,\cdots,x_{r+p}):M^r\longrightarrow \mathbb{R}^{r+p}_+$ be
an immersed submanifold. We define
\emph{rotational submanifold over $u$} in $\mathbb{R}^{m+p}$ as
\begin{equation*}
\begin{split}
&f:M^r\times \mathbb{S}^{m-r}\longrightarrow \mathbb{R}^{m+p},\\
&f(u,\phi)=(x_1,\cdots,x_{r+p-1},x_{r+p}\phi).
\end{split}
\end{equation*}
where $\phi:\mathbb{S}^{m-r}\longrightarrow \mathbb{R}^{m-r+1}$ is the standard sphere.
\end{Definition}

\begin{PROPOSITION}\label{33}
Let $u=(x_1,\cdots,x_{r+p}):M^r\longrightarrow \mathbb{R}^{r+p}_+$ be
an immersed submanifold.
Then the rotational submanifold
$f:M^r\times \mathbb{S}^{m-r}\longrightarrow \mathbb{R}^{m+p}$
is a Wintgen ideal submanifold if and only if $u$ is a minimal Wintgen ideal submanifold.
\end{PROPOSITION}

\begin{proof}
Let $\mathbb{R}^{r+p+1}_1$ be the
Lorentz space with inner product
\[
\langle y,y\rangle=-y_1^2+y_2^2+\cdots+y_{r+p+1}^2,\;\; y=(y_1,\cdots,y_{r+p+1}).
\]
Let $\mathbb{H}^{r+p}=\{y\in R^{r+p+1}_1|\langle y,y\rangle =-1,y_1>0\}$ be the hyperbolic space.
Introduce isometry $\tau:\mathbb{R}^{r+p}_+\longrightarrow \mathbb{H}^{r+p}$ as below:
\begin{equation}\label{eq-tau}
\tau(x_1,\cdots,x_{r+p})=\left(\frac{1+x_1^2+\cdots+x_{r+p}^2}{2x_{r+p}},
\frac{1-x_1^2-\cdots-x_{r+p}^2}{2x_{r+p}},\frac{x_1}{x_{r+p}},\cdots,
\frac{x_{r+p-1}}{x_{r+p}}\right)~.
\end{equation}

Let $\eta^1,\cdots,\eta^p$ be the unit normal vectors of $u$ in $\mathbb{R}^{r+p}_+$.
Write
$\eta^i=(\eta^i_1,\cdots,\eta^i_{r+p}).$
Since $\eta^i$ is the unit normal vector, then
\[
\frac{(\eta_1^i)^2+\cdots+(\eta_{r+p}^i)^2}{x_{r+p}^2}=1,~~~1\leq i\leq p.
\]
Thus the unit normal vector of  $f$ in $\mathbb{R}^{m+p}$ is
\[
\xi_i=\frac{1}{x_{r+p}}(\eta^i_1,\cdots,\eta^i_{r+p}\phi).
\]
The first fundamental form of $u$ is
\[
I_u=\frac{1}{x_{r+p}^2}(dx_1\cdot dx_1+\cdots+dx_{r+p}\cdot dx_{r+p}).
\]
The second fundamental form of $u$ is
\[
II^i_u=-\langle \tau_*(du),\tau_*(d\eta^i)\rangle
=\frac{1}{x_{r+p}^2}(dx_1\cdot d\eta^i_1+\cdots
+dx_{r+p}\cdot d\eta^i_{r+p})-\frac{\eta^i_{r+p}}{x_{r+p}}I_u.
\]
Now we can write out the first and the second fundamental forms
of $f$:
\[
I=dx\cdot dx=x_{r+p}^2(I_u+I_{S^{m-r}}),
~~II^i=x_{r+p}II^i_u-\eta^i_{r+p}(I_u+I_{s^{m-r}}),
\]
where $I_{\mathbb{S}^{m-r}}$ is the standard metric of $\mathbb{S}^{m-r}$.
 This completes the proof.
\end{proof}

\begin{Remark}
The M\"{o}bius position vector $Y:M^r\times \mathbb{S}^{m-r}\longrightarrow \mathbb{R}^{m+p+2}_1$ of the rotational submanifold $f$ is
\begin{equation*}
Y=\rho_0(\frac{1+|u|^2}{2x_{r+p}},
\frac{1-|u|^2}{2x_{r+p}},\frac{x_1}{x_{r+p}},\cdots,
\frac{x_{r+p-1}}{x_{r+p}},\phi),
\end{equation*}
where $\rho_0=\frac{m}{m-1}(|II_u|^2-mH_u^2): M^r\longrightarrow R$, and $\phi:\mathbb{S}^{m-r}\longrightarrow \mathbb{S}^{m-r}$ is the identity map.
Since $(\frac{1+|u|^2}{2x_{r+p}},
\frac{1-|u|^2}{2x_{r+p}},\frac{x_1}{x_{r+p}},\cdots,
\frac{x_{r+p-1}}{x_{r+p}})=\tau(u):M^r\to \mathbb{H}^{r+p}$,
then the M\"{o}bius position vector of the rotational submanifold $f$ is
\begin{equation}\label{mop3}
Y=\rho_0(\tau(u),\phi):M^r\times \mathbb{S}^{m-r}\to \mathbb{H}^{r+p}\times \mathbb{S}^{m-r}\subset \mathbb{R}^{m+p+2}_1,
\end{equation}
where $\rho_0\in C^{\infty}(M^r)$ and $\phi:\mathbb{H}^{m-r}\to \mathbb{H}^{m-r}$ is the identity map.
\end{Remark}

From (\ref{mop1}), (\ref{mop2}) and (\ref{mop3}), we have
\begin{PROPOSITION}\label{redu}
Let $f:M^m\to \mathbb{R}^{m+p}$ be
an immersed submanifold without umbilical points.\\
(1) If there exists a submanifold $u:M^r\to \mathbb{R}^{r+p}$ such that the M\"{o}bius position vector of $f$ is
\begin{equation*}
\begin{split}
&Y=\rho_0\left(\frac{1+|u|^2+|y|^2}{2},\frac{1-|u|^2-|y|^2}{2},
u,y\right)\\
&Y:M^r\times \mathbb{R}^{m-r}\longrightarrow \mathbb{R}^{r+p+2}_1\times \mathbb{R}^{m-r}\subset \mathbb{R}^{m+p+2}_1,
\end{split}
\end{equation*}
where $\rho_0\in C^{\infty}(M^r)$, and $y:\mathbb{R}^{m-r}\longrightarrow \mathbb{R}^{m-r}$ is the identity map.
Then $f$ is a cylinder over $u$.\\
(2) If there exists a submanifold $u:M^r\to \mathbb{S}^{r+p}$ such that the M\"{o}bius position vector of $f$ is
\begin{equation*}
Y=\rho_0(id,u):\mathbb{H}^{m-r}\times M^r\to \mathbb{H}^{m-r}\times \mathbb{S}^{r+p}\subset \mathbb{R}^{m+p+2}_1,
\end{equation*}
where $\rho_0\in C^{\infty}(M^r)$ and $id:\mathbb{H}^{m-r}\to \mathbb{H}^{m-r}$ is the identity map.
Then $f$ is a cone over $u$.\\
(3) If there exists a submanifold $u:M^r\to \mathbb{R}^{r+p}_+$ such that the M\"{o}bius position vector of $f$ is
\begin{equation*}
Y=\rho_0(\tau(u),\phi):M^r\times \mathbb{S}^{m-r}\to \mathbb{H}^{r+p}\times \mathbb{S}^{m-r}\subset \mathbb{R}^{m+p+2}_1,
\end{equation*}
where $\rho_0\in C^{\infty}(M^r)$, $\phi:\mathbb{S}^{m-r}\to \mathbb{S}^{m-r}$ is the identity map, and $\tau(u)$ is defined as in \eqref{eq-tau}.
Then $f$ is the rotational submanifold over $u$.
\end{PROPOSITION}

\section{Proof of the Main theorem}

A submanifold $f:M^{m}\to
\mathbb{R}^{m+p}$ is a Wintgen ideal submanifold if and only if, at
each point of $M^{m}$, there is a suitable frame such that the second
fundamental form has the form (\ref{form1}). If $\mu_0=0$ in (\ref{form1}), then the Wintgen ideal submanifold is totally umbilical submanifold.
Next we consider non-umbilical Wintgen ideal submanifolds, that is $\mu_0\neq 0$ on $M^m$ and $m\geq 3$.

 Since $\mu_0\neq 0$, we can choose an orthonormal basis $\{E_1,\cdots,E_m\}$ of $T_xM^m$ with respect to the M\"{o}bius metric $g$ and an orthonormal basis $\{\xi_1,\cdots,\xi_p\}$ of $T_x^{\bot}M^m$ such that the coefficients of the M\"obius second fundamental form $B$
has the form
\begin{equation}\label{eq-equality}
B^{1}=
\begin{pmatrix}
0 & \mu & 0 & \cdots & 0\\
\mu & 0 & 0 & \cdots & 0\\
0  & 0 & 0 & \cdots & 0\\
\vdots & \vdots & \vdots & \ddots & \vdots\\
0  & 0 & 0 & \cdots & 0
\end{pmatrix},~~~~~~
B^{2}=
\begin{pmatrix}
\mu & 0 & 0 & \cdots & 0\\
0 & -\mu & 0 & \cdots & 0\\
0  & 0 & 0 & \cdots & 0\\
\vdots & \vdots & \vdots & \ddots & \vdots\\
0  & 0 & 0 & \cdots & 0
\end{pmatrix},
\end{equation}
and $ B^{\alpha}=0$ for all $\alpha\ge 3$.
By \eqref{equa7}, the norm of $B$ is constant and
$\mu=\sqrt{\frac{m-1}{4m}}$. Clearly the distribution $\mathbb{D}=span\{E_1,E_2\}$ is well-defined.

For convenience we adopt the convention
below on the range of indices:
\[
1\le i,j,k,l\le m,~~3\le a,b,c\le m,~~3\le \alpha,\beta,\gamma \le
p.
\]

Then our assumption means that except
\[
B^{1}_{12}=B^{2}_{21}=B^{2}_{11}=-B^{2}_{22}=\mu=\sqrt{\frac{m-1}{4m}},
\]
any other coefficient of the M\"obius second fundamental form vanishes. In particular,
\begin{equation}\label{eq-equality2}
B^{1}_{11}=B^1_{22}=B^{1}_{aj}=0,\hskip 5pt
B^2_{12}=B^{2}_{21}=B^{2}_{bj}=0,\hskip 5pt B^{\alpha}_{ij}=0.
\end{equation}

First we compute the covariant derivatives of $B^{r}_{ij}$. Denote
\begin{equation}\label{eq-theta}
\theta\triangleq 2\omega_{12}+\theta_{12}.
\end{equation}
Since the M\"obius second fundamental form $B$ has the form (\ref{eq-equality}), we have
\begin{equation}\label{bb1}
B^{\delta}_{ab,k}=0, 1\leq \delta\leq p, 1\leq k\leq m,~~B^{\alpha}_{1a,i}=0,  B^{\alpha}_{2a,i}=0.
\end{equation}

\begin{equation}\label{bb2}
\omega_{2a}=\sum_i\frac{B^1_{1a,i}}{\mu}\omega_i=-\sum_i\frac{B^2_{2a,i}}{\mu}\omega_i,~~ \omega_{1a}=\sum_i\frac{B^1_{2a,i}}{\mu}\omega_i=\sum_i\frac{B^2_{1a,i}}{\mu}\omega_i.
\end{equation}

\begin{equation}\label{bb3}
\begin{split}
&\theta=\sum_i\frac{-B^1_{11,i}}{\mu}\omega_i=\sum_i\frac{B^1_{22,i}}{\mu}\omega_i=\sum_i\frac{B^2_{12,i}}{\mu}\omega_i,\\
&B^1_{12,i}=0,~~B^2_{11,i}=B^2_{22,i}=0.
\end{split}
\end{equation}

It follows from (3)
that
$$ C^{\alpha}_1=B^{\alpha}_{aa,1}-B^{\alpha}_{a1,a}=0;
C^{\alpha}_2=B^{\alpha}_{aa,2}-B^{\alpha}_{a2,a}=0.$$
$$ C^{\alpha}_a=B^{\alpha}_{11,a}-B^{\alpha}_{1a,1}=B^{\alpha}_{11,a},C^{\alpha}_a=B^{\alpha}_{22,a}-B^{\alpha}_{2a,2}=B^{\alpha}_{22,a},$$
Since $\sum_iB^{\delta}_{ii,k}=0, 1 \leq \delta \leq p,1\leq k\leq m $, we have
$$C^{\alpha}=0.$$
From (\ref{bb2}) and (\ref{bb3}), we obtain
$$B^1_{2a,2}=B^1_{22,a}=B^2_{1a,2},~~B^2_{1a,1}=0,B^2_{2a,2}=0.$$
This implies that $C^1_a=0,~~C^2_a=B^2_{11,a}=B^2_{22,a}$, thus $C^2_a=0.$

The other coefficients of $\{C^{r}_j\}$ are obtained similarly as
below:
\begin{equation}\label{cc1}
\begin{split}
C^{1}_1=-B^{1}_{1a,a}=-\mu\omega_{2a}(e_a),~~C^{2}_2=-B^{2}_{2a,a}=\mu\omega_{2a}(e_a), \\
C^{1}_2=-B^{1}_{2a,a}=-\mu\omega_{1a}(e_a), ~~C^{2}_1=-B^{2}_{1a,a}=-\mu\omega_{1a}(e_a).
\end{split}
\end{equation}
In particular we have
\begin{equation}
C^{1}_1=-C^{2}_2,~~C^{1}_2=C^{2}_1.
\end{equation}
The covariant derivative of coefficients of $C$ are
\begin{equation}\label{c12}
C^1_{1,i}=-C^2_{2,i}, ~~C^1_{2,i}=C^2_{1,i},~~C^{\alpha}_{a,i}=0.
\end{equation}

From (\ref{bb2}) and (\ref{bb3}), we write out the connection forms
\begin{equation}\label{conne}
\begin{split}
&\theta=\sum_i\frac{-B^1_{11,i}}{\mu}\omega_i=\sum_i\frac{B^1_{22,i}}{\mu}\omega_i=\sum_i\frac{B^2_{12,i}}{\mu}\omega_i,\\
&\omega_{1a}=\frac{B^1_{2a,2}}{\mu}\omega_2+\frac{B^1_{2a,a}}{\mu}\omega_a=\frac{B^2_{1a,2}}{\mu}\omega_2+\frac{B^2_{1a,a}}{\mu}\omega_a,\\
&\omega_{2a}=\frac{B^1_{1a,1}}{\mu}\omega_1+\frac{B^1_{1a,a}}{\mu}\omega_a=\frac{-B^2_{2a,1}}{\mu}\omega_1-\frac{B^2_{2a,a}}{\mu}\omega_a.
\end{split}
\end{equation}
Now we use the assumption that
the distribution $\mathbb{D}=span\{E_1,E_2\}$ is integrable, which says
$$d\omega_a\equiv0, mod\{\omega_a\}.$$
From (\ref{conne}), we obtain
\begin{equation*}
B^1_{11,a}=-B^1_{22,a}=-B^2_{12,a}=0.
\end{equation*}

\begin{equation}\label{conne1}
\theta=\frac{C^1_1}{\mu}\omega_1-\frac{C^1_2}{\mu}\omega_2,
\omega_{1a}=\frac{-C^1_2}{\mu}\omega_a, ~~\omega_{2a}=\frac{-C^1_1}{\mu}\omega_a.
\end{equation}
\begin{equation}\label{curv6}
\begin{split}
-\frac{1}{2}\sum_{ij}R_{1aij}\omega_i\wedge\omega_j=-\sum_i\frac{C^1_{2,i}}{\mu}\omega_i\wedge\omega_a+\frac{(C^1_1)^2+(C^1_2)^2}{\mu^2}\omega_1\wedge\omega_a,\\
-\frac{1}{2}\sum_{ij}R_{2aij}\omega_i\wedge\omega_j=-\sum_i\frac{C^1_{1,i}}{\mu}\omega_i\wedge\omega_a+\frac{(C^1_1)^2+(C^1_2)^2}{\mu^2}\omega_2\wedge\omega_a,\\
-\sum_{ij}(R_{12ij}+\frac{1}{2}R^{\perp}_{12ij})\omega_i\wedge\omega_j=\sum_i\frac{C^1_{1,i}}{\mu}\omega_i\wedge\omega_1-\sum_i\frac{C^1_{2,i}}{\mu}\omega_i\wedge\omega_2\\
+\left[\frac{(C^1_1)^2+(C^1_2)^2}{\mu^2}
+\sum_{\alpha}\frac{(B^{\alpha}_{11,2})^2
+(B^{\alpha}_{22,1})^2}{\mu^2}\right]\omega_1\wedge\omega_2,
\end{split}
\end{equation}

From (\ref{curv6}), we obtain
\begin{equation}\label{abc}
\begin{split}
&R_{1a1a}=A_{11}+A_{aa}=\frac{C^1_{2,1}}{\mu}-\frac{(C^1_1)^2+(C^1_2)^2}{\mu^2},\\
&R_{2a2a}=A_{22}+A_{aa}=\frac{C^1_{1,2}}{\mu}-\frac{(C^1_1)^2+(C^1_2)^2}{\mu^2}\\
&R_{1a2a}=\frac{C^1_{2,2}}{\mu},~~R_{2a1a}=\frac{C^1_{1,1}}{\mu},~~R_{1a12}=A_{2a}=0,~~R_{2a12}=-A_{1a}=0,\\
&R_{121a}=A_{2a}=\frac{C^1_{1,a}}{\mu},~~R_{122a}=-A_{1a}=\frac{C^1_{2,a}}{\mu}.
\end{split}
\end{equation}
The equations (\ref{abc}) implies that
$$L:=A_{aa}=A_{bb}, A_{1a}=A_{2a}=A_{ab}=0, a\neq b.$$

Define new frame vectors
\begin{equation*}
\begin{split}
&\hat{Y}=\frac{-(C^1_1)^2-(C^1_2)^2}{2\mu^2}Y+N-\frac{C^1_1}{\mu}Y_2-\frac{C^1_2}{\mu}Y_1,\\
&\eta_1=Y_1+\frac{1}{\mu}C_2^1Y,~~\eta_2=Y_2+\frac{1}{\mu} C^1_1Y,~~K=2L+\frac{(C^1_1)^2+(C^1_2)^2}{\mu^2}.
\end{split}
\end{equation*}
Then we have the moving frame $\{Y,\hat{Y},\eta_1,\eta_2,Y_3,\cdots,Y_m,\xi_1,\xi_2,\xi_3,\cdots,\xi_p\}$, such that
$\mathbb{R}^{m+p+2}=span\{Y,\hat{Y}\}\bigoplus span\{\eta_1,\eta_2,Y_3,\cdots,Y_m,\xi_1,\xi_2,\xi_3,\cdots,\xi_p\}$, $\langle Y,\hat{Y}\rangle =1$ and
$\{\eta_1,\eta_2,Y_3,\cdots,Y_m,\xi_1,\xi_2,\xi_3,\cdots,\xi_p\}$ are orthonormal vector fields.

Using (\ref{abc}) and (\ref{curv6}), we have
\begin{equation}\label{frame1}
\begin{split}
d\xi_1=-\mu\omega_1\eta_2-\mu\omega_2\eta_1+\sum_{s=1}^p\theta_{1s}\xi_s,\\
d\xi_2=-\mu\omega_1\eta_1+\mu\omega_2\eta_2+\sum_{s=1}^p\theta_{2s}\xi_s,\\
d\xi_{\alpha}=-\theta_{1\alpha}\xi_1-\theta_{2\alpha}\xi_2+\sum_{\beta}\theta_{\alpha\beta}\xi_{\beta}.
\end{split}
\end{equation}
\begin{equation}\label{frame2}
\begin{split}
&d\eta_1=\left[\omega_{12}-\frac{C^1_1}{\mu}\omega_1
+\frac{C^1_2}{\mu}\omega_2\right]\eta_2
+\omega_1\left(\frac{K}{2}Y-\hat{Y}\right)
+\mu\omega_2\xi_1+\mu\omega_1\xi_2,\\
&d\eta_2=-\left[\omega_{12}-\frac{C^1_1}{\mu}\omega_1
+\frac{C^1_2}{\mu}\omega_2\right]\eta_1
+\omega_2\left(\frac{K}{2}Y-\hat{Y}\right)
+\mu\omega_1\xi_1-\mu\omega_2\xi_2,\\
&d\left(\frac{K}{2}Y-\hat{Y}\right)=K[\omega_1\eta_1+\omega_2\eta_2]
+\left[\frac{C^1_2}{\mu}\omega_1+\frac{C^1_1}{\mu}\omega_2\right]
\left(\frac{K}{2}Y-\hat{Y}\right).
\end{split}
\end{equation}
\begin{equation}\label{k}
E_1(K)=2\frac{C^1_2}{\mu}K,~~
E_2(K)=2\frac{C^1_1}{\mu}K,~~
E_a(K)=0.
\end{equation}
From (\ref{frame1}) and (\ref{frame2}), we know that the subspace
$$V=span\{(\frac{K}{2}Y-\hat{Y}),\eta_1,\eta_2,\xi_1,\xi_2,\cdots,\xi_p\}$$
is parallel along $M^m$. The orthogonal complement $V^{\perp}$ also is parallel along $M^m$. In fact,
$$V^{\perp}=span\{(\frac{K}{2}Y+\hat{Y}),Y_3,\cdots,Y_m\}.$$
 Using (\ref{abc}) and (\ref{curv6}), we can obtain
\begin{equation}\label{t}
d(\frac{K}{2}Y+\hat{Y})=\left(\frac{C^1_2}{\mu}\omega_1
+\frac{C^1_1}{\mu}\omega_2\right)(\frac{K}{2}Y+\hat{Y})
+K\sum_a\omega_aY_a.
\end{equation}

Clearly, the distribution $\mathbb{D}^{\perp}=span\{E_3,\cdots,E_m\}$ also is integrable. From (\ref{frame1}) and  (\ref{frame2}),
we know that the mean curvature spheres $\xi_1,\xi_2$ induce $2$-dimensional submanifolds in the de sitter space $\mathbb{S}_1^{m+p+1}$
$$\xi_1,\xi_2:M^2=M^m/F\longrightarrow \mathbb{S}_1^{m+p+1},$$
where fibers $F$ are integral submanifolds of distribution $\mathbb{D}^{\perp}$. In other words, $\xi_1,\xi_2$ form $2$-parameter family of $(m+p-1)$-spheres enveloped
by $f:M^m\longrightarrow \mathbb{R}^{m+p}$.

Since $\langle\frac{K}{2}Y-\hat{Y},\frac{K}{2}Y-\hat{Y}\rangle =-\langle\frac{K}{2}Y+\hat{Y},\frac{K}{2}Y+\hat{Y}\rangle =-K$ satisfies a linear first-order PDE (\ref{k}), we see that $K\equiv 0$ or $K\neq 0$ on the connected open set of $M^m$. Thus there are three possibilities for the induced metric on the fixed subspace $V,~V^{\perp}\subset \mathbb{R}^{m+p+2}_1$.

\vspace{1mm}
\noindent{\bf Case 1}. $K<0$ on $M^m$; $V$ is a fixed space-like subspace, $V^{\bot}$ is a fixed Lorentz subspace in $\mathbb{R}^{m+p+2}_1$.
We can assume that $V=\mathbb{R}^{3+p}, ~~V^{\bot}=\mathbb{R}^{m-1}_1$. From (\ref{frame1}), (\ref{frame2}) and (\ref{k}), we know
$$u=\frac{1}{\sqrt{-K}}(\frac{K}{2}Y-\hat{Y}):M^2\to \mathbb{S}^{2+p}.$$
On the other hand, the equation (\ref{t}) implies that
$$\phi=\frac{1}{\sqrt{-K}}(\frac{K}{2}Y+\hat{Y}):\mathbb{H}^{m-2}\to \mathbb{R}^{m-1}_1$$
is the embedding of the hyperbolic space $\mathbb{H}^{m-2}$ in $\mathbb{R}^{m-1}_1$.
Then
\begin{equation*}
Y=2\sqrt{-K}(u,\phi):M^2\times \mathbb{H}^{m-2}\to \mathbb{S}^{2+p}\times \mathbb{H}^{m-2}\subset \mathbb{R}^{m+p+2}_1,
\end{equation*}
where $2\sqrt{-K}\in C^{\infty}(M^2)$ and $\phi:\mathbb{H}^{m-2}\to \mathbb{H}^{m-2}$ is a identity map.
From Proposition (\ref{redu}), we know that $f$ is a cone over $u:M^2\to \mathbb{S}^{2+p}.$

\vspace{1mm}
\noindent{\bf Case 2}. $K>0$ on $M^m$; $V$ is a fixed Lorentz subspace, $V^{\bot}$ is a fixed space-like subspace in $\mathbb{R}^{m+p+2}_1$.
We can assume that $V=\mathbb{R}^{3+p}_1, ~~V^{\bot}=\mathbb{R}^{m-1}$. From (\ref{frame1}), (\ref{frame2}) and (\ref{k}), we know
$$u=\frac{1}{\sqrt{K}}(\frac{K}{2}Y-\hat{Y}):M^2\to \mathbb{H}^{2+p}.$$
On the other hand, the equation (\ref{t}) implies that
$$\phi=\frac{1}{\sqrt{K}}(\frac{K}{2}Y+\hat{Y}):\mathbb{S}^{m-2}\to \mathbb{R}^{m-1}$$
is the embedding of the sphere $\mathbb{S}^{m-2}$ in $\mathbb{R}^{m-1}$.
Then
\begin{equation*}
Y=2\sqrt{K}(u,\phi):M^2\times \mathbb{S}^{m-2}\to \mathbb{H}^{2+p}\times \mathbb{S}^{m-2}\subset \mathbb{R}^{m+p+2}_1,
\end{equation*}
where $2\sqrt{K}\in C^{\infty}(M^2)$ and $\phi:\mathbb{S}^{m-2}\to \mathbb{S}^{m-2}$ is the identity map.
From Proposition (\ref{redu}), we know that $f$ is the rotational submanifold over $u:M^2\to \mathbb{H}^{2+p}.$

\vspace{1mm}
\noindent{\bf Case 3}. $K=0$ on $M^m$. From (\ref{t}), we can assume that $\hat{Y}=e^{\varrho}(-1,1,0,\cdots,0)$, where $\varrho\in C^{\infty}(M^2)$.
On the other hand, $V,~~ V^{\perp}$ are two fixed spaces endowed with a degenerate inner product. we can assume
that $V=span\{(\frac{K}{2}Y-\hat{Y}),\eta_1,\eta_2,\xi_1,\xi_2,\cdots,\xi_p\}=\mathbb{R}^{3+p}_0,~~V^{\bot}=R^{m-1}_0.$
We write vector $v\in \mathbb{R}^{3+p}_0$ and $w\in \mathbb{R}^{m-1}_0$ by
$$u=(u_0,-u_0,u_1,\cdots,u_{p+2},0,\cdots,0),~w=(w_0,-w_0,0,\cdots,0,w_1,\cdots,w_{m-2});$$
and we write
$$e^{\sigma}Y=\left(\frac{1+|f|^2}{2},\frac{1-|f|^2}{2},f\right), ~~f=(u_1,\cdots,u_{p+2},w_1,\cdots,w_{m-2})\in \mathbb{R}^{m+p}.$$
From (\ref{frame1}) and (\ref{frame2}), we know that
$$u=(u_1,\cdots,u_{p+2}):M^2\to \mathbb{R}^{2+p}$$
is an immersed surface, and
$$w=(w_1,\cdots,w_{m-2}):\mathbb{R}^{m-2}\to \mathbb{R}^{m-2}$$
is the identity map.
From Proposition (\ref{redu}), we know that $f$ is the cylinder over $u:M^2\to \mathbb{R}^{2+p}.$

Combining Proposition \ref{31}, \ref{32} and \ref{33}, we complete the proof to Theorem~\ref{the1}.\qed

\begin{Remark}
From (\ref{frame1}) and (\ref{frame2}), we obtain
\begin{equation}\label{hat0}
\begin{split}
&dY=-(\frac{C^1_2}{\mu}\omega_1+\frac{C^1_1}{\mu}\omega_2)Y+\omega_1\eta_1+\omega_2\eta_2+\sum_a\omega_aY_a,\\
&d\hat{Y}=(\frac{C^1_2}{\mu}\omega_1+\frac{C^1_1}{\mu}\omega_2)\hat{Y}+\frac{K}{2}\sum_a\omega_aY_a
-\frac{K}{2}(\omega_1\eta_1+\omega_2\eta_2).
\end{split}
\end{equation}
Thus we have
\begin{equation}\label{hat}
\langle d\hat{Y},d\hat{Y}\rangle =\frac{K^2}{4}\langle dY,dY\rangle =\frac{K^2}{4}g.
\end{equation}
Let $\hat{f}:M^m\to R^{m+p}$ be an immersed submanifold such that the M\"{o}bius position vector is $\hat{Y}$.
From $\langle \hat{Y},\xi_1\rangle =\cdots=\langle \hat{Y},\xi_p\rangle =0$ and (\ref{hat0}), we know that the submanifold  $\hat{f}:M^m\to R^{m+p}$
envelops the mean curvature spheres $\{\xi_1,\cdots,\xi_p\}$. And if $\hat{f}$ is an immersed submanifold, then
$\hat{f}$ also is a Wintgen ideal submanifold and is conformal to $f$. This is analogous to the duality phenomenon for Willmore surfaces in $\mathbb{S}^3$, but simpler than that. Here $\hat{f}$ either differ from $f$ by an antipodal map of the sphere in Case~1, or by an inversion/reflection with respect to the boundary at infinity of the hyperbolic space in Case~2, or degenerate to the single point at
infinity of the Euclidean space in Case~3.
\end{Remark}

{\bf Acknowledgements:} The authors thank Dr. Yuquan Xie for helpful discussions.

\vspace{5mm} \noindent Tongzhu Li,
{\small\it Department of Mathematics, Beijing Institute of
Technology, Beijing 100081, People's Republic of China.
e-mail:{\sf litz@bit.edu.cn}}

\vspace{5mm} \noindent Xiang Ma,
{\small\it School of Mathematical Sciences, Peking University,
Beijing 100871, People's Republic of China.
e-mail: {\sf maxiang@math.pku.edu.cn}}

\vspace{5mm} \noindent Changping Wang,
{\small\it School of Mathematics and Computer Science,
Fujian Normal University, Fuzhou 350108, People's Republic of China.
e-mail: {\sf cpwang@fjnu.edu.cn}}

\end{document}